\documentclass[journal]{IEEEtran}

\usepackage{cite}

\ifCLASSINFOpdf
\usepackage[pdftex]{graphicx}
\graphicspath{{../pdf/}{../jpeg/}}
\DeclareGraphicsExtensions{.pdf,.jpeg,.png}
\else
\fi

\usepackage[cmex10]{amsmath}

\usepackage{amssymb}

\usepackage{comment}
\usepackage{array,multirow}
\usepackage{tikz}

\usepackage{accents}

\usepackage{bbm}
\usepackage{dsfont}
\usepackage[bb=boondox]{mathalfa}

\usepackage{algorithm}
\usepackage{algorithmic}

\usepackage{amsthm}
\usepackage{enumerate}

\theoremstyle{definition}
\newtheorem{definition}{Definition}
\newtheorem{theorem}{Theorem}
\newtheorem{corollary}{Corollary}
\newtheorem{lemma}{Lemma}
\newtheorem{proposition}{Propostion}
\newtheorem{remark}{Remark}
\newtheorem*{problem formulation}{Problem formulation}
\usepackage{dsfont}

\usepackage{mathtools}

\usepackage{color,soul}

\begin{document}
\title{Robustness against Disturbances in Power Systems under Frequency Constraints}
\author{Dongchan Lee, Liviu Aolaritei, Thanh Long Vu and Konstantin~Turitsyn
\thanks{D. Lee, L. Vu and K. Turitsyn are with the Department of Mechanical Engineering, Massachusetts Institute of Technology, Cambridge, MA 02139, USA (email: dclee@mit.edu; longvu@mit.edu; turitsyn@mit.edu). 

L. Aolaritei is with the Automatic Control Laboratory, ETH Zurich,
Zurich 8092, Switzerland (e-mail:aliviu@ethz.ch). 

The work was supported by funding from the U.S. Department of Energys Office of Electricity as part of the DOE Grid Modernization Initiative.}
}

\maketitle

\begin{abstract}
The wide deployment of renewable generation and the gradual decrease in the overall system inertia make modern power grids more vulnerable to transient instabilities and unacceptable frequency fluctuations. Time-domain simulation-based assessment of the system robustness against uncertain and stochastic disturbances is extremely time-consuming. {\color{black} In this paper, we develop an alternative approach, which has its roots in the input-output stability analysis for Lur'e systems. Our approach consists of a mathematically rigorous characterization of the external disturbances that the power system is transiently stable and the frequency constraints are not violated.} The derived certificate is efficiently constructed via convex optimization and is shown to be non-conservative for different IEEE test cases. 
\end{abstract}


\IEEEpeerreviewmaketitle

\section{Introduction}
Transient stability assessment is one of the most computationally challenging security assessment procedures carried out by the system operators \cite{kundur94, pavella12, machowski11}. In addition to transient stability, the operators are required to maintain the system frequency close to the nominal values of $50$ or $60$ Hz \cite{kundur04}. The grid is equipped with under-frequency load shedding (UFLS) relays, as well as under-frequency and over-frequency generation protection relays to ensure that the frequency regulation is met \cite{miller11}. Traditionally, the frequency deviations during faults were suppressed by the turbine speed governors and by the natural inertia of the generators. However, in recent years, the primary frequency response capabilities have steadily declined in many power grids, e.g., the Eastern Interconnection \cite{ingleson10} in the US. This decline in response results in deeper frequency nadir, which, in turn, increases the risk of unintended disconnection of units and cascading outages.

{\color{black} In addition to the decrease in the power system inertia, the source of disturbance has significantly increased, due to the higher penetration of renewables and distributed generators.} Typical disturbances could include nearly-instant switching events, such as load shedding and generation tripping, or continuous changes, such as varying power output from wind turbines \cite{miller11, bernal16}. One of the most common causes of frequency rise is the near simultaneous tripping of more than one generators. As a consequence, it is very important to be able to efficiently quantify the critical disturbance levels that the grid can withstand at any given operating condition.

{\color{black} In recent years, there have been many efforts to assess the transient stability of power systems under operational (e.g., frequency) constraints. These studies can be divided into three main groups. The first group proposes numerical simulations under stochastic disturbances, where the output trajectory is computed for a given realization of the disturbance \cite{dong12,milano13,Papadopoulos2017}. Time-domain simulations yield high fidelity assessments when the disturbance and the operating conditions are known exactly. However, when there is limited information about the disturbance, the assessment may require large number of simulations. The second group is based on reachability analysis, where the output trajectories are bounded inside the reachable set \cite{Chen2012,Choi2017,lee2017robust,Althoff2014,zhang18}. While some of these formulations allow differential-algebraic equations to model the power grid dynamics, they rely on the approximation of the dynamics via linearization or Taylor-series expansion \cite{Chen2012,Choi2017}. The works in \cite{Choi2017,lee2017robust} give tight time-dependent bounds on the output, but they require solving an optimization problem at every time step. The third and final group is based on Input-to-State Stability (ISS) \cite{Aolaritei2018,WEITENBERG20181} analysis. ISS provides a powerful rigorous approach to tackle such a problem, however finding a Lyapunov function that renders this approach non-conservative is in general very difficult.}

In this paper, we propose a tractable method for finding the bound on the maximum magnitude of the disturbance that the grid can withstand without violating the frequency constraints. This will allow system operators to certify that the grid is robust against an entire class of magnitude-bounded disturbances. The disturbances are only characterized by their magnitude, and therefore instant step changes such as switching or tripping are also considered in the proposed analysis. 

The methodology proposed to solve this problem builds on the input-output stability analysis \cite{Sontag1999} for nonlinear systems, which we specialize to systems written in a Lur'e form representation \cite{khalil96}. A Lur'e system is a linear dynamical system with a nonlinear static state feedback, where the nonlinearity is sector bounded by two linear functions \cite{kokotovic01,zames66,vu16_lyap}. The Lur'e system representation with local sector bounded nonlinearity has been recently applied to power systems for finding the region of attraction \cite{vu16_lyap, Aolaritei2018}. In our formulation, the power system is seen as an input-output map from the disturbance to the frequency of the generators. Small-gain arguments are then used to assess the input-output stability of the system under output constraints.

{\color{black} The main contributions of our paper are as follows. First, we define the notions of Constrained Input Bounded Output (CIBO) stability and Constrained Input Constrained Output (CICO) stability. These definitions extend the well-known Bounded Input Bounded Output (BIBO) stability notion to consider constraints on both the input and the output. The term CICO stability has also appeared in the context of filter design in \cite{Lev-Ari1991}. Second, we provide a certificate on the disturbance magnitude such that the resulting generator frequencies are constrained within some operational limits provided by the system operators. Our result guarantees that the system is robust against all possible realizations of magnitude-bounded disturbances. Third, we show that finding the maximum disturbance magnitude can be solved via convex optimization when the generator angle separation constraint is imposed. The ability to quickly and efficiently assess the potential impact of disturbance provides a significant advantage to our method in the real-time operation of power grids compared to the other approaches in the literature. }

The rest of the paper is organized as follows. In Chapter II, we present the system model, together with it's Lur'e representation, and we mathematically formulate the problem. In Section III, we define three notions of input-output stability, for which we present sufficient conditions in Section IV. In Section V, we build on the stability analysis previously developed to formulate a convex optimization problem of finding the maximum magnitude of the admissible disturbance. The results are numerically validated in Section V on the IEEE 9-bus and 39-bus test cases. Finally, Section VI concludes the paper.

\section{System Model and Problem Formulation}

The power grid is represented as an undirected graph $\mathcal{A}(\mathcal{N},\mathcal{E})$, where $\mathcal{N}=\{1,2,...,n\}$ is the set of buses, and $\mathcal{E}\subseteq\mathcal{N}\times\mathcal{N}$ is the set of transmission lines connecting the buses. Let $\ell=|\mathcal{E}|$. The indices $\mathcal{G} = \{1,...,m\}$ denote the generators, and $\mathcal{L} = \{m+1,...,m+n\}$ denote the loads. {\color{black} 
Let $E\in\mathbf{R}^{n\times \ell}$ denote the incidence matrix of the graph}. Moreover, let $\mathbf{0}$ and $I$ denote the zero matrix and the identity matrix of appropriate dimensions, respectively. Finally, given a matrix $A\in \mathbf{R}^{n \times n}$, its spectral radius is denoted by $\rho(A)$.

\subsection{Power System Model}

The structure-preserving second-order swing equation is used to model the power system dynamics:
\begin{equation}
\begin{aligned}
M_k\ddot{\delta}_k+D_k\dot{\delta}_k+\sum_{(k,j)\in\mathcal{E}}\phi_{kj}\sin(\delta_{kj})&=p_{k}, &\forall k\in\mathcal{G} \\
D_k\dot{\delta}_k+\sum_{(k,j)\in\mathcal{E}}\phi_{kj}\sin(\delta_{kj})&=P_{L,k}, &\forall k\in\mathcal{L}
\end{aligned}
\label{eqn_swing}
\end{equation}
where $M_k$ and $D_k$ are the inertia and damping coefficients of the generator $k$, respectively. $p_k$ and $P_{L,k}$ are the mechanical power at generator $k$ and load $k$, respectively. Moreover, $\phi_{kj}=b_{kj} V_k V_j$, where $b_{kj}$ is the susceptance of the transmission line $(k,j)$, and $V_k$ is the voltage magnitude at bus $k$, which we assume constant. Finally, $\delta_{kj}$ denotes the phase difference between bus $k$ and bus $j$, i.e., $ \delta_{kj}=\delta_k-\delta_j$.

{\color{black} In addition to the grid dynamics, we consider the turbine governor dynamics, which introduce delay in the primary frequency control response. The delayed response often leads to greater excursion from the nominal grid frequency. This effect is captured by the following first order turbine governor model:}
\begin{equation}
T_k\dot{p}_k+p_k+\frac{1}{R_k}\dot{\delta}_k=P_{G,k}, \ k\in\mathcal{G},
\label{eqn_governor}
\end{equation}
where $P_{G,k}$ is the scheduled power injection at bus $k$, $T_k$ is the governor time constant, and $R_k$ is the droop coefficient.
To write the system model \eqref{eqn_swing} and \eqref{eqn_governor} in vector form, the following notation is introduced. Let $\delta_G$ and $\delta_L$ be the vectors obtained by stacking the scalars $\delta_k$, for $k \in \mathcal{G}$, and $\delta_k$, for $k \in \mathcal{L}$, respectively. Moreover, let $\delta =\begin{bmatrix} \delta_G^T & \delta_L^T\end{bmatrix}^T$. Similarly, let $p$, $P_G$ and $P_L$ be the vectors obtained by stacking the scalars $p_k$, $P_{G,k}$, for $k \in \mathcal{G}$, and $P_{L,k}$ for $k \in \mathcal{L}$, respectively. 
{\color{black} Let $M$, $D_G$, $D_L$ and $\Phi$ be the diagonal matrices containing the elements $M_k$, $D_k$, for $k \in \mathcal{G}$, $D_k$, for $k \in \mathcal{L}$, and $\Phi_{kj}$, for $(k,j) \in \mathcal{E}$, on their diagonal, respectively.} Finally, let $E=\begin{bmatrix} E_G^T & E_L^T\end{bmatrix}^T$, where the subscripts $G$ and $L$ correspond to the generator and load buses, respectively.

Consider now the disturbance vector $u=\begin{bmatrix} u_G^T & u_L^T\end{bmatrix}^T.$ The system model \eqref{eqn_swing} and \eqref{eqn_governor} can be rewritten in the following form: 
\begin{equation}
\begin{aligned}
M\ddot{\delta}_G+D_G\dot{\delta}_G+E_G \Phi\sin(E^T\delta)&=p \\
D_L\dot{\delta}_L+E_L \Phi\sin(E^T\delta)&=P_L+u_L \\
T\dot{p}+p+R^{-1}\dot{\delta}_G&=P_G+u_G.
\end{aligned}
\label{eqn_swing_vector}
\end{equation}
This simple formulation of the disturbance could incorporate a rich variety of uncertainty scenarios, such as load shedding, generation tripping, and stochastic fluctuations in the power output from wind turbines.

\subsection{Lur'e System Representation}

In the following, the system \eqref{eqn_swing_vector} will be rewritten as a Lur'e system, i.e., as an interconnection of a linear dynamical system with a nonlinear static state feedback. As it will be shown in this paper, the Lur'e system, together with the efficient bounding of the nonlinearity between linear functions, heavily simplifies the analysis of the nonlinear power systems.

The system model \eqref{eqn_swing_vector} can be written in a state-space representation. For $u_L=\mathbf{0}$ and $u_G=\mathbf{0}$, let $\delta^*$ and $\dot{\delta}=\mathbf{0}$ represent the equilibrium point of \eqref{eqn_swing_vector}, with generator power injection $p^*$. Then, we define the state of the system as $x = \begin{bmatrix} x_1^T & x_2^T & x_3^T & x_4^T\end{bmatrix}^T$, with $x_1 = \delta_G - \delta_G^*$, $x_2 = \dot{\delta}_G$, $x_3 = \delta_L - \delta_L^*$, and $x_4 = p - p^*$.

Now let $z=E^T\delta-E^T\delta^*$ be the phase difference on each transmission line subtracted by its equilibrium, and $y$ be the vector containing the frequencies of the generators $y=\dot{\delta}_G$. Finally, let $\varphi^* = E^T\delta^*$, and  $v=\sin(\varphi^*+z)-\sin(\varphi^*)-\text{diag}(\cos(\varphi^*))z$.
With these new variables, the system \eqref{eqn_swing_vector} can be written in the Lur'e form $\dot{x}=Ax+B_v v+B_u u$ as follows:
\begin{equation}
\begin{aligned}
\dot{x}&=\begin{bmatrix} \mathbf{0} & I & \mathbf{0} & \mathbf{0} \\ A_{21} & -M^{-1}D_G & A_{23} & M^{-1} \\ A_{31} & \mathbf{0} & A_{33} & \mathbf{0} \\ \mathbf{0} & -R^{-1}T^{-1} & \mathbf{0} & -T^{-1} \end{bmatrix}x \\
& \hskip 3em +\begin{bmatrix} \mathbf{0} \\ -M^{-1} E_G \Phi \\ -D_L^{-1} E_L \Phi \\ \mathbf{0} \end{bmatrix}v +\begin{bmatrix} \mathbf{0} & \mathbf{0} \\ \mathbf{0} & \mathbf{0} \\ \mathbf{0} & D_L^{-1} \\ T^{-1} & \mathbf{0} \end{bmatrix}u
\end{aligned}
\label{eqn_lure}
\end{equation}
with
$$\begin{aligned}
A_{21}&=-M^{-1}E_G \Phi\text{diag}(\cos\varphi^*)E_G^T \\
A_{23}&=-M^{-1}E_G \Phi\text{diag}(\cos\varphi^*)E_L^T \\
A_{31}&=-D_L^{-1}E_L \Phi\text{diag}(\cos\varphi^*)E_G^T \\
A_{33}&=-D_L^{-1}E_L \Phi\text{diag}(\cos\varphi^*)E_L^T.
\end{aligned}$$

The complete model can be compactly written as
\begin{subequations}
\begin{align}
\dot{x}&=Ax+B_v v+B_u u \label{eqn_fullmodel_x} \\
v&=\sin(\varphi^*+z)-\sin\varphi^*-\text{diag}(\cos\varphi^*)z \label{eqn_fullmodel_v} \\
y&=\begin{bmatrix}\mathbf{0} & I & \mathbf{0} & \mathbf{0}\end{bmatrix} x=C_yx \label{eqn_fullmodel_y} \\
z&=\begin{bmatrix}E_G^T & \mathbf{0} & E_L^T & \mathbf{0}\end{bmatrix}x=C_zx \label{eqn_fullmodel_w}.
\end{align}
\label{eqn_fullmodel}
\end{subequations}
The matrix $A$ in \eqref{eqn_lure} was obtained by linearization of the system \eqref{eqn_swing_vector} around the equilibrium point $x=\mathbf{0}$. The vector $v$ represents the static nonlinear feedback of the state $x$, i.e., $v=\psi(z)=\psi(C_z x)$.

\begin{figure}[tb]
	\centering
	\includegraphics[width=1.5in]{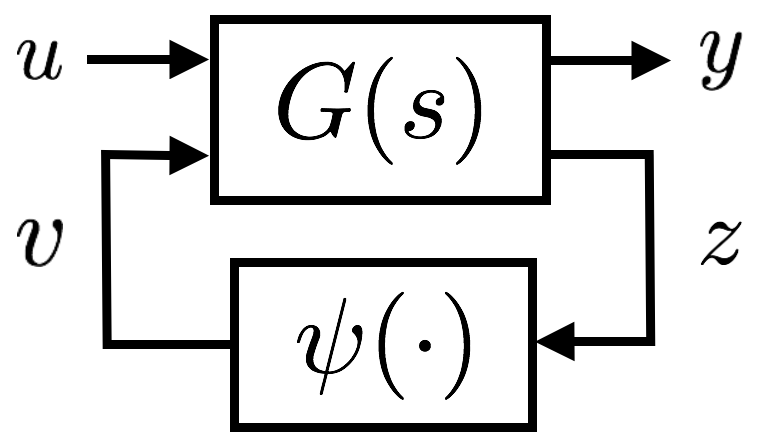}
	\caption{Lur'e system representation of the power system dynamics in $G(s)$ and the nonlinear components in $\psi(\cdot)$.}
	\label{fig_system_w_uncertainty}
\end{figure}

Let the transfer function matrix $G(s)$ represent the linear dynamics in Laplace domain. Then the Lur'e system \eqref{eqn_lure} can be graphically represented as in Figure \ref{fig_system_w_uncertainty}. Following this representation of the system, the transfer function matrix $G$ can be divided into four blocks:
\begin{equation}
G(s)=\begin{bmatrix} G_{y,u}(s) & G_{y,v}(s) \\ G_{z,u}(s) & G_{z,v}(s) \end{bmatrix}
\label{eqn_G}
\end{equation}
where each block of transfer matrix can be computed by $G_{i,j}(s)=C_i(sI-A)^{-1}B_j$, with $i \in \{y,z\}$ and $j \in \{u,v\}$. {\color{black} This representation of the system implies that the initial condition of the system is at the equilibrium (i.e. $x_0=\mathbf{0}$)}. Given the system model described in this section, the problem can be formulated as follows. 



\subsection{Problem Formulation}
\label{problemformulation}
Consider the power system model \eqref{eqn_fullmodel}, containing the additive magnitude-bounded disturbance $u$. The analysis carried out in this paper concentrates on finding the maximum bound on the magnitude of the disturbance such that the generators remain synchronized, and some imposed constraints on the frequencies of the generators are never violated. 

In order to quantify the magnitude of the disturbance $u$, we propose the following element-wise infinity norm.

\begin{definition}
\label{def:infinitynorm}
Let $u(t)\in\mathbf{R}^n$. Its element-wise $\mathcal{L}$-infinity norm, which we denote by $\lvert u \rvert_{\mathcal{L}_{\infty}^n}\in\mathbf{R}^n$, is defined as
\begin{equation}
\big[\lvert u \rvert_{\mathcal{L}_{\infty}^n}\big]_i=\sup_{t\geq0}|u_i(t)|
\end{equation}
where $\big[\lvert u \rvert_{\mathcal{L}_{\infty}^n}\big]_i$ and $u_i$ are the $i$-th entries of $\lvert u \rvert_{\mathcal{L}_{\infty}^n}$ and $u$, respectively.
\end{definition}

\begin{remark}
{\color{black} The element-wise $\mathcal{L}$-infinity {\color{black} generalizes} the standard $\mathcal{L}_\infty$ norm of $u$, defined as $\lVert u\rVert_{\mathcal{L}_{\infty}}=\max_i (\sup_{t\geq0}|u_i(t)|)$. The proposed element-wise norm allows us to represent different magnitudes of disturbance at the individual buses, rather than bounding them uniformly. This fact will be exploited in Section \ref{sec:optimization}, where an optimization problem will be formulated to compute the maximum magnitude of the admissible disturbance entering at each bus. To avoid any confusion, we denoted the element-wise $\mathcal{L}$-infinity norm of an $n$-dimensional signal by $\mathcal{L}_{\infty}^n$, where the superscript $n$ should remind the reader that $\lvert \cdot \rvert_{\mathcal{L}_{\infty}^n}$ is an $n$-dimensional vector.}
\end{remark}

The problem can be now mathematically formulated as follows.

\begin{problem formulation} {\color{black} Consider the power system \eqref{eqn_swing} written in the Lur'e form \eqref{eqn_fullmodel}, with initial condition $x_0=\mathbf{0}$. The objective of our problem is to find the maximum bound $\bar{u} \in \mathbf{R}^n$ on the disturbance such that if $\lvert u \rvert_{\mathcal{L}_{\infty}^n} \leq \bar{u}$, the following two conditions hold:
\begin{itemize}
    \item[(i)] $\exists \; \bar{z}$  such that $\lvert z \rvert_{\mathcal{L}_{\infty}^{\ell}} \leq \bar{z}$
    \item[(ii)] $\lvert y \rvert_{\mathcal{L}_{\infty}^m} \leq \bar{y}$.
\end{itemize}}

The first condition, which introduces a constraint $\bar{z}$ on the difference in angles between adjacent buses, prevents the angular separation of the generators, and therefore ensures that the generators remain synchronized during the transient dynamics. The second condition ensures that the frequency constraints, defined by $\bar{y}$, are not violated.

\end{problem formulation}

{\color{black}
\section{Definitions and Preliminaries}

Let $y = H u$ define an input-output relation, where $H$ is an operator that specifies the output $y$ in terms of the input $u$. In the following we will introduce three types of input-output stability notions for the operator $H$ with respect to the element-wise infinity norm $\lvert \cdot \rvert_{\mathcal{L}_{\infty}^{\cdot}}$. The first type of input-output stability is the Bounded Input Bounded Output stability, which is defined as follows.

\begin{definition}[\textbf{Bounded Input Bounded Output}]
\label{BIBO}
The operator $H$ is BIBO stable if for every input constraint $\bar{u}$, if $\lvert u \rvert_{\mathcal{L}_{\infty}^n}\leq \bar{u}$, then the output $\lvert y \rvert_{\mathcal{L}_{\infty}^m}$ is bounded.
\end{definition}

 Notice that it is not always possible to meet such a condition, especially for nonlinear systems, such as the power grid. We now define the second type of input-output stability, namely the Constrained Input Bounded Output (CIBO) stability.

\begin{definition}{\textbf{(Constrained Input Bounded Output)}}
\label{CIBO}
The operator $H$ is CIBO stable if there exists an input constraint $\bar{u}$, such that for every input $u$ with $\lvert u \rvert_{\mathcal{L}_{\infty}^n}\leq \bar{u}$, the output $\lvert y \rvert_{\mathcal{L}_{\infty}^m}$ is bounded.
\end{definition}

Recall from Section \ref{problemformulation} that we want to find the maximum bound on the magnitude of the disturbance such that the generators remain synchronized, and the constraints on the frequencies of the generators are never violated. We formalize this concept into the third and last type of input-output stability, Constrained Input Constrained Output (CICO) stability.

\begin{definition}{\textbf{(Constrained Input Constrained Output)}}
\label{CICO}
The operator $H$ is CICO stable if given an output constraint $\bar{y}$, there exists an input constraint $\bar{u}$, such that for every input with $u$ with $\lvert u \rvert_{\mathcal{L}_{\infty}^n}\leq \bar{u}$, the output satisfies $\lvert y \rvert_{\mathcal{L}_{\infty}^m}\leq\bar{y}$.
\end{definition}

In Section \ref{sec:inputoutput} we will propose conditions under which the system \eqref{eqn_fullmodel} is BIBO, CIBO, and CICO stable.

\begin{remark}
{\color{black} Under normal operating conditions, the linear dynamics $G(s)$ in the Lur'e system \eqref{eqn_fullmodel} is BIBO stable. Notice that any uniform shift in the angles $\delta$ defines another equilibrium point, and therefore the matrix $A$ cannot be Hurwitz, but only marginally stable. However, it can be shown that the eigenvalues of $A$ with zero real part do not appear in $G(s)$ due to pole-zero cancellation. The cancellation occurs because the angles do not appear in the output of the linear system, but only the angle differences. In the power system literature, this is known as small-signal stability, which is a necessary condition for the system to be transiently stable.}
\end{remark}

Let $y = H u$ be a BIBO stable system. We define its gain to be a non-negative constant matrix  $\gamma_H\in\mathbf{R}^{m\times n}$ such that
\begin{equation}
     \lvert y \rvert_{\mathcal{L}_{\infty}^m}\leq \gamma_H \lvert u \rvert_{\mathcal{L}_{\infty}^n}.
\end{equation}

When the input is bounded, i.e., $\lvert u \rvert_{\mathcal{L}_{\infty}^n}\leq \bar{u}$, we denote the gain matrix by $\gamma_H(\bar{u})$ to remind that it is a function of the domain parametrized by $\bar{u}$.

For a BIBO stable linear system, where the operator $H$ corresponds to the transfer function $G(s)$ in equation \eqref{eqn_G}, the gain matrix $\gamma_G$ can be computed using the following lemma.

\begin{lemma}
Given a BIBO stable linear system with transfer function $G(s)$, the $ij$ element of the gain matrix $\gamma_G$ can be computed as
\begin{equation}
\gamma_{G,ij}={\color{black}\lVert G_{ij}\rVert_{\mathcal{L}_1}}
\label{eqn_gain1}
\end{equation}
with $\lVert G_{ij}\rVert_{\mathcal{L}_1}=\int_{-\infty}^{\infty}|h_{ij}(\tau)|d\tau$, where $h_{ij}$ is the impulse response of $G_{ij}$.
\label{lemma_BIBO}
\end{lemma}
\begin{proof}
For the $i$-th element of the output vector,
{\color{black}
$$\begin{aligned}
|y_i(t)|\leq \sum_j\bar{u}_j\int_{-\infty}^{\infty}|h_{ij}(\tau)|d\tau=\sum_j\lVert G_{ij}\rVert_{\mathcal{L}_1}\bar{u}_j.
\end{aligned}$$}
\end{proof}

The matrix $\gamma_G$, can be divided, according to \eqref{eqn_G}, into
\begin{equation}
\gamma_G=\begin{bmatrix}
        \gamma_{y,u} & \gamma_{y,v} \\
        \gamma_{z,u} & \gamma_{z,v} \\
     \end{bmatrix},
\end{equation}
 where $\gamma_{y,u}\in\mathbf{R}^{m\times n}$, $\gamma_{y,v}\in\mathbf{R}^{m\times \ell}$, $\gamma_{z,u}\in\mathbf{R}^{\ell\times n}$, and $\gamma_{z,v}\in\mathbf{R}^{\ell\times \ell}$ are the gain matrices computed as shown in Lemma \ref{lemma_BIBO}.

Consider now the nonlinear component, given by $v = \psi(z)$. Since it is decentralized, i.e., $v_i = \psi_i(z_i)$ $\forall i\in \{1,\ldots,\ell\}$, the gain matrix $\gamma_\psi$ is a diagonal matrix. The diagonal element of $\gamma_\psi$ in the position $\{i,i\}$ is equal to:
\begin{equation}
\gamma_{\psi,ii}=\sup_{z_i}\bigg|\frac{v_i}{z_i}\bigg|,
\label{eqn_gain_phi1}
\end{equation}
and direct substitution of equation \eqref{eqn_fullmodel_v} results in
\begin{equation}
\gamma_{\psi,ii}=\sup_{z_i}\bigg|\frac{\sin(\varphi^*_i+z_i)-\sin\varphi^*_i}{z_i}-\cos\varphi^*_i\bigg|
\label{eqn_gain_phi2}
\end{equation}
which is finite for bounded phase angles. Therefore, the nonlinear component $\psi(\cdot)$ is BIBO stable.
}

\section{Input-Output Stability Analysis}
\label{sec:inputoutput}

In this section we will establish the mathematical framework for the analysis and assessment of the system stability under the additive disturbance $u$. The proposed framework combines the input-output stability approach with the sector bounds on the nonlinearity $v$ in the Lur'e system to propose a novel small-gain theorem based on the element-wise $\mathcal{L}$-infinity norm $\lvert \cdot \rvert_{\mathcal{L}_{\infty}^{\cdot}}$.

Given the computed gain matrices of the system, the following inequalities hold:
\begin{subequations}
\begin{align}
\lvert y \rvert_{\mathcal{L}_{\infty}^m}&\leq\gamma_{y,u} \lvert u \rvert_{\mathcal{L}_{\infty}^n}+\gamma_{y,v}\lvert v \rvert_{\mathcal{L}_{\infty}^\ell} \label{eqn_plant_y_gain} \\
\lvert z \rvert_{\mathcal{L}_{\infty}^{\ell}}&\leq\gamma_{z,u} \lvert u \rvert_{\mathcal{L}_{\infty}^n}+\gamma_{w,v}\lvert v \rvert_{\mathcal{L}_{\infty}^\ell} \label{eqn_plant_w_gain} \\
\lvert v \rvert_{\mathcal{L}_{\infty}^\ell}&\leq\gamma_\psi \lvert z \rvert_{\mathcal{L}_{\infty}^{\ell}} \label{eqn_nonlinearity_gain}
\end{align}
\label{eqn_plant_nonlinearity}
\end{subequations}
The gain matrices are non-negative, i.e., $\gamma_{i,j}\geq 0 $, $ \forall \; i,j$. Using this property, we state the following lemma, which will be important in the proofs of the subsequent results of this paper.
\begin{lemma}
Given the positive matrices $\gamma_{w,v}$ and $\gamma_\psi$, the following three conditions are equivalent:
\begin{itemize}
\item[(i)] $\rho(\gamma_{z,v}\gamma_\psi)<1$
\item[(ii)] $(I-\gamma_{z,v}\gamma_\psi)^{-1}\geq 0$
\item[(iii)] There exists $x\geq0$ such that $(I-\gamma_{z,v}\gamma_\psi) x>0$
\end{itemize}
\label{lemma_Zmatrix}
\end{lemma}
\begin{proof}
The proof is based on the properties of $Z$ and $M$-matrices. A matrix is a $Z$-matrix if its off-diagonal elements are non-positive, and it is an $M$-matrix if it is a $Z$-matrix and its eigenvalues have non-negative real parts.  First the matrix $I-\gamma_{z,v}\gamma_\psi$ is a $Z$-matrix since the gain matrices are non-negative. Now notice that $\rho(\gamma_{z,v}\gamma_\psi)<1$ if and only if the eigenvalues of $I-\gamma_{z,v}\gamma_\psi$ have positive real parts, which is the definition of a nonsingular $M$-matrix. Given that $I-\gamma_{z,v}\gamma_\psi$ is a nonsingular $M$-matrix, condition (i), (ii), and (iii) are equivalent \cite{plemmons77}. 
\end{proof} 

The theory of $Z$ and $M$-matrices also appears in the power systems literature in the analysis of the steady-state voltage stability of distribution networks \cite{aolaritei18b}.

\begin{remark}
Since the matrix $\gamma_{z,v}\gamma_\psi$ is nonnegative, it has a real eigenvalue equal to its spectral radius $\rho(\gamma_{z,v}\gamma_\psi)$ \cite{bullo17}.
\end{remark}

In the next theorem, we present the condition under which the power system is BIBO stable.

\begin{theorem}[\textbf{Small-Gain Theorem}]
 The system \eqref{eqn_fullmodel} is BIBO stable if the gain matrices $\gamma_G$ and $\gamma_\psi$ are finite, and $\rho(\gamma_{z,v}\gamma_\psi)<1$.
\label{thm_sgt}
\end{theorem}
\begin{proof}
By substituting Equation \eqref{eqn_plant_w_gain} into Equation \eqref{eqn_nonlinearity_gain} and rearranging, we have
$$(I-\gamma_{z,v}\gamma_\psi)\lvert z \rvert_{\mathcal{L}_{\infty}^{\ell}}\leq\gamma_{z,u} \lvert u \rvert_{\mathcal{L}_{\infty}^n}.$$
Since $\rho(\gamma_{z,v}\gamma_\psi)<1$, Lemma \ref{lemma_Zmatrix} guarantees that $(I-\gamma_{z,v}\gamma_\psi)^{-1}\geq 0$. As such,
$$\lvert z \rvert_{\mathcal{L}_{\infty}^{\ell}}\leq(I-\gamma_{z,v}\gamma_\psi)^{-1}\gamma_{z,u} \lvert u \rvert_{\mathcal{L}_{\infty}^n}.$$
The output can be bounded by
$$\begin{aligned}
\lvert y \rvert_{\mathcal{L}_{\infty}^m}&\leq\gamma_{y,u} \lvert u \rvert_{\mathcal{L}_{\infty}^n}+\gamma_{y,v}\lvert v \rvert_{\mathcal{L}_{\infty}^\ell} \\
&\leq\gamma_{y,u} \lvert u \rvert_{\mathcal{L}_{\infty}^n}+\gamma_{y,v}\gamma_\psi \lvert z \rvert_{\mathcal{L}_{\infty}^{\ell}} \\
&\leq\big[\gamma_{y,u}+\gamma_{y,v}\gamma_\psi (I-\gamma_{z,v}\gamma_\psi)^{-1}\gamma_{z,u}\big] \lvert u \rvert_{\mathcal{L}_{\infty}^n}.
\end{aligned}$$
Therefore, the system is BIBO stable.
\end{proof}

\begin{remark}
Theorem \ref{thm_sgt} ensures more than just BIBO stability. Indeed, the last inequality in the proof implies that there exists a non-negative constant gain matrix $\gamma_H=\big[\gamma_{y,u}+\gamma_{y,v}\gamma_\psi (I-\gamma_{z,v}\gamma_\psi)^{-1}\gamma_{z,u}\big]$ such that
\begin{equation}
\lvert y \rvert_{\mathcal{L}_{\infty}^m} \leq \gamma_H \lvert u \rvert_{\mathcal{L}_{\infty}^n}.
\label{eqn_infty_gain}
\end{equation} 
and therefore the system is finite gain $\mathcal{L}_{\infty}^{\cdot}$ stable.
\end{remark}

Theorem \ref{thm_sgt} presents a novel small-gain theorem, defined for the element-wise $\mathcal{L}$-infinity norm $\lvert \cdot \rvert_{\mathcal{L}_{\infty}^{\cdot}}$. {\color{black} The small-gain condition ensures BIBO stability, which guarantees that the output is bounded for any bounded input. However, a power grid does not have a globally stable equilibrium, and therefore the boundedness of the output cannot be guaranteed for every bounded input. Thus, for such a system the small-gain condition is not satisfied in general, for any bounded input. However, if the magnitude of the disturbance is constrained by some appropriate $\bar{u}$, the output could be bounded. This reasoning is further explained in the following.}

{\color{black} The condition in Theorem \ref{thm_sgt} is not satisfied for an arbitrary nonlinear gain matrix $\gamma_\psi$. Indeed, since $\rho(\gamma_{z,v}\gamma_\psi)<1$, it results that for fixed linear gain matrix $\gamma_{z,v}$, there exists a limit on the magnitude of $\gamma_\psi$ such that our system is BIBO stable. This can be deduced from the fact that $\gamma_\psi$ is a non-negative diagonal matrix, and therefore the spectral radius $\rho(\gamma_{z,v}\gamma_\psi)$ is a strictly increasing function in $\gamma_\psi$}. Recall that $z$ corresponds to the phase differences deviation from the phase differences at the equilibrium, i.e., $z = E^T \delta - E^T \delta^*$. Let $\bar{z}$ be some magnitude bound on $z$, i.e., $\lvert z \rvert_{\mathcal{L}_{\infty}^{\ell}} \leq \bar{z}$. Now $\gamma_\psi$ is function of $\bar{z}$, i.e., $\gamma_\psi = \gamma_\psi(\bar{z})$, and that larger $\bar{z}$ results in larger $\gamma_\psi(\bar{z})$ (see Figure \ref{fig_sector_bound_cos}). As a consequence, the condition in Theorem \ref{thm_sgt} could be satisfied for some $\bar{z}$.

\begin{figure}[tb]
	\centering
	\includegraphics[width=2.8in]{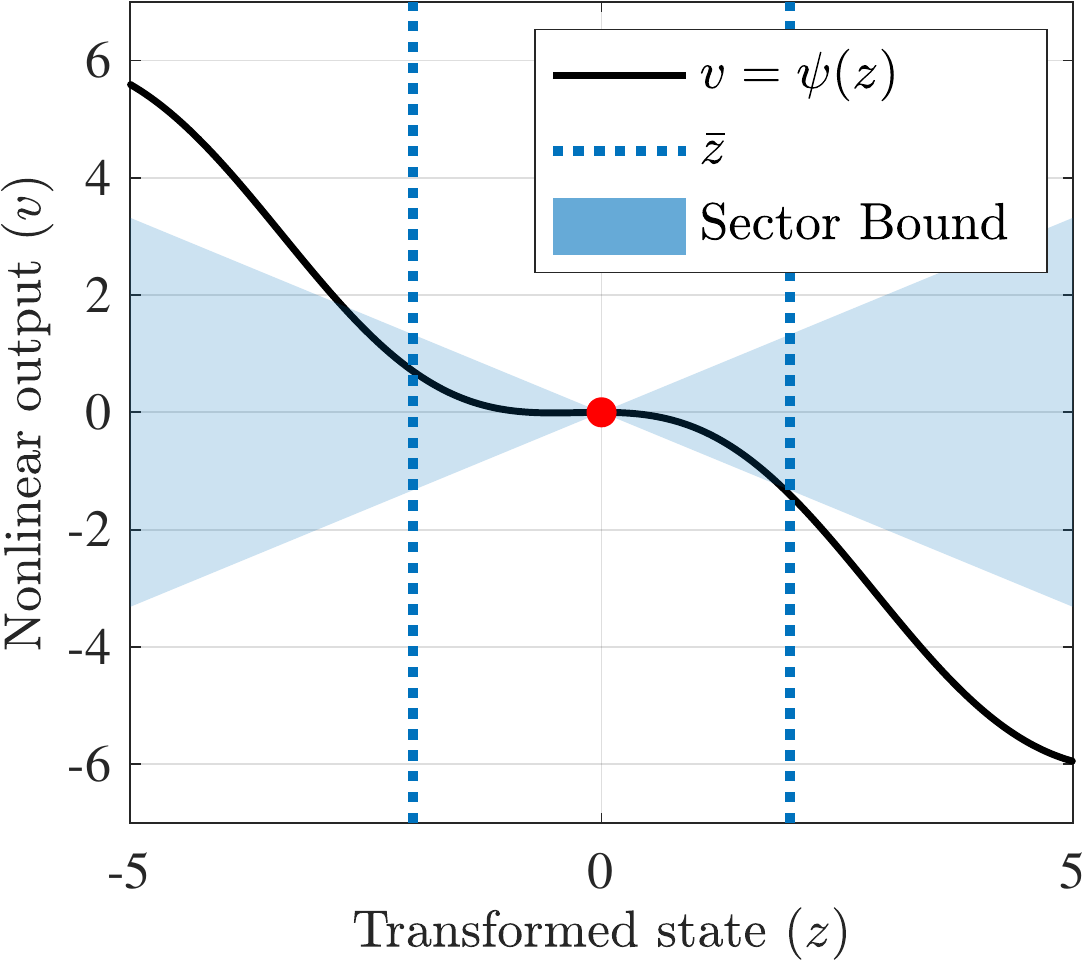}
	\caption{Sector bound for $v=\psi(z)=\sin(z+\varphi^*)-\cos(\varphi^*)z$.}
	\label{fig_sector_bound_cos}
\end{figure}

This observation is exploited in the following theorem, where a sufficient condition for the CIBO stability of our system is presented.

\begin{theorem}
Let $\bar{u}$ be a bound on the magnitude of the input, i.e., $\lvert u \rvert_{\mathcal{L}_{\infty}^n} \leq \bar{u}$. If  $\gamma_G$ and $\gamma_\psi$ are finite, and if there exist $\bar{u}$ and $\bar{z}$ satisfying
\begin{equation}
\gamma_{z,u}\bar{u}<(I-\gamma_{z,v}\gamma_\psi(\bar{z}))\bar{z},
\label{cond_CIBO}
\end{equation}
then the system \eqref{eqn_fullmodel} is CIBO stable and $\lvert z \rvert_{\mathcal{L}_{\infty}^{\ell}}\leq\bar{z}$.
\label{thm_CIBO}
\end{theorem}
\begin{proof}
Since the gain matrix is a positive matrix and $\bar{u}\geq0$, $(I-\gamma_{z,v}\gamma_\psi(\bar{z}))\bar{z}>\gamma_{z,u}\bar{u}\geq0$ with $\bar{z}\geq0$. From Lemma \ref{lemma_Zmatrix}, $\rho(\gamma_{z,v}\gamma_\psi)<1$, and by using Theorem \ref{thm_sgt}, the system is BIBO stable for $\lvert u \rvert_{\mathcal{L}_{\infty}^n} \leq \bar{u}$.
Substituting condition \eqref{cond_CIBO} and Equation \eqref{eqn_nonlinearity_gain} into Equation \eqref{eqn_plant_y_gain}, we have
$$\begin{aligned}\lvert z \rvert_{\mathcal{L}_{\infty}^{\ell}}&\leq\gamma_{z,u} \lvert u \rvert_{\mathcal{L}_{\infty}^n}+\gamma_{z,v} \lvert v \rvert_{\mathcal{L}_{\infty}^\ell} \\
&\leq(I-\gamma_{z,v}\gamma_\psi(\bar{z}))\bar{z}+\gamma_{z,v}\gamma_\psi(\bar{z}) \lvert z \rvert_{\mathcal{L}_{\infty}^{\ell}}. \end{aligned}$$
By rearranging,
$$(I-\gamma_{z,v}\gamma_\psi(\bar{z}))\lvert z \rvert_{\mathcal{L}_{\infty}^{\ell}}\leq(I-\gamma_{z,v}\gamma_\psi(\bar{z}))\bar{z} \\$$
Now, $I-\gamma_{z,v}\gamma_\psi(\bar{z})$ is inverse-positive from Lemma \ref{lemma_Zmatrix}, so $\lvert z \rvert_{\mathcal{L}_{\infty}^{\ell}} \leq \bar{z}$.
\end{proof}

\begin{remark}
{\color{black} Theorem \ref{thm_CIBO} provides a local small-gain condition over the domain $\lvert z \rvert_{\mathcal{L}_{\infty}^{\ell}}\leq\bar{z}$. If the condition \eqref{cond_CIBO} is satisfied for all $\bar{z}$, then it is equivalent to the small-gain condition from Theorem \ref{thm_sgt}, and the system is BIBO stable.}
\end{remark}
This remark can be directly observed from Lemma \ref{lemma_Zmatrix}. This inequality condition is a different representation of the small-gain condition, but further exploits the fact that $\gamma_\psi$ can be a function of $\bar{z}$. There is a natural trade-off based on the value of $\bar{z}$. The nonlinear gain $\gamma_\psi$ increases as $\bar{z}$ increases, which makes it difficult to meet the small-gain condition. On the other hand, small $\bar{z}$ imposes a stricter bound on the phase difference on the transmission lines. This trade-off is represented as the product of $I-\gamma_{z,v}\gamma_\psi(\bar{z})$ and $\bar{z}$, which are monotonically decreasing and linearly increasing functions of $\bar{z}$, respectively.

Now, in order to enforce the generator frequency constraints, we need to impose an additional condition that will guarantee the CICO stability. This is presented in the following theorem.

\begin{theorem}
Let $\bar{u}$ be a bound on the magnitude of the input, i.e., $\lvert u \rvert_{\mathcal{L}_{\infty}^n} \leq \bar{u}$. If  $\gamma_G$ and $\gamma_\psi$ are finite, and if there exist $\bar{u}$ and $\bar{z}$ such that
\begin{equation}
\begin{aligned}
&\gamma_{z,u}\bar{u}<(I-\gamma_{z,v}\gamma_\psi(\bar{z}))\bar{z} \\
&\gamma_{y,u}\bar{u}+\gamma_{y,v}\gamma_\psi(\bar{z})\bar{z} \leq \bar{y}
\end{aligned}
\label{cond_CICO}
\end{equation}
then the system \eqref{eqn_fullmodel} is CICO stable. Moreover, we have $\lvert z \rvert_{\mathcal{L}_{\infty}^{\ell}}\leq\bar{z}$ and $\lvert y \rvert_{\mathcal{L}_{\infty}^m}\leq\bar{y}$.
\label{thm_CICO}
\end{theorem}
\begin{proof}
From Theorem \ref{thm_CIBO}, the first condition in \eqref{cond_CICO} ensures $\lvert z \rvert_{\mathcal{L}_{\infty}^{\ell}} \leq \bar{z}$. Moreover, the substitution of the condition in this theorem and Equation \eqref{eqn_nonlinearity_gain} into Equation \eqref{eqn_plant_y_gain} results in
$$\lvert y \rvert_{\mathcal{L}_{\infty}^m}\leq\gamma_{y,u} \lvert u \rvert_{\mathcal{L}_{\infty}^n}+\gamma_{y,v}\gamma_\psi \lvert z \rvert_{\mathcal{L}_{\infty}^{\ell}} \leq \gamma_{y,u} \bar{u}+\gamma_{y,v}\gamma_\psi \bar{z}\leq\bar{y}.$$
\end{proof}

{\color{black} The inequalities proposed in Theorem \ref{thm_CICO} provide a sufficient condition for CICO stability. Condition \eqref{cond_CICO} will be used in the next section as a constraint in an optimization problem that computes the maximum admissible disturbance magnitude.}

\section{Computation of the Disturbance Bound}
\label{sec:optimization}

 In the following, an optimization problem is formulated to find the bound $\bar{u}$ on the disturbance such that the frequencies of the generators remain inside the operational limits. Given a potential disturbance $u$, the system operator only needs to check that $\lvert u \rvert_{\mathcal{L}_{\infty}^n}\leq\bar{u}$ is satisfied to ensure that the generator frequency constraints are not violated. The input-output stability framework developed in Theorem \ref{thm_CICO} will be used to solve this problem.

The first step in doing so is to derive an explicit expression for the gain of nonlinear component $\gamma_\psi$. Recall that $\gamma_\psi$ is function of $\bar{z}$:
\begin{equation}
\gamma_{\psi,ii}(\bar{z}_i)=\sup_{|z_i|\leq\bar{z_i}}\bigg|\frac{\sin(z_i+\varphi^*_i)-\sin(\varphi^*_i)}{z_i}-\cos(\varphi^*_i)\bigg|
\label{eqn_gain_phi3}
\end{equation}
where $\varphi^*=E^T\delta^*$.

 {\color{black}In the following corollary, we derive an analytical expression for the gain of the nonlinear components $\gamma_{\psi,ii}(\bar{z}_i)$, for angle deviation constraints that are of practical interest.}

\begin{corollary}
Let $\bar{z}$ be a bound on the angle difference between generators and $\varphi^*=E^T\delta^*$ be such that $|\varphi^*_i|+\bar{z}_i\leq \pi, |\varphi^*_i|\leq \frac{\pi}{2}$ $\forall i$. Then,
\begin{equation}
\label{corollary_phigain}
\gamma_{\psi,ii}(\bar{z}_i)\leq\cos|\varphi^*_i|-\frac{\sin(|\varphi^*_i|+\bar{z}_i)-\sin|\varphi^*_i|}{\bar{z}_i}.
\end{equation}
\end{corollary}

\begin{proof}
From Equation \eqref{eqn_gain_phi2} and given $|\varphi^*_i|\leq \frac{\pi}{2}$, we have
$$\begin{aligned} 
&\gamma_{\psi,ii}(\bar{z}_i)=\sup_{|z_i|\leq\bar{z_i}}\bigg|\frac{\sin(z_i+\varphi^*_i)-\sin(\varphi^*_i)}{z_i}-\cos(\varphi^*_i)\bigg| \\ 
&=\sup_{|z_i|\leq\bar{z}_i}\bigg|\frac{\sin z_i-z_i}{z_i}\cos\varphi^*_i+\frac{\cos z_i-1}{z_i}\sin\varphi^*_i\bigg| \\ 
&\leq\sup_{|z_i|\leq\bar{z}_i}\frac{|z_i|-\sin |z_i|}{|z_i|}\cos|\varphi^*_i|+\frac{1-\cos |z_i|}{|z_i|}\sin|\varphi^*_i|
\end{aligned}$$
Moreover, the function inside the supremum is increasing monotonically with respect to $z_i$ for $|\varphi^*_i|+\bar{z}_i\leq \pi$. Therefore, the inequality \eqref{corollary_phigain} holds true.
\end{proof}

 {\color{black} The analytical expression for the gain of the nonlinearity \eqref{eqn_gain_phi3} will be used in the conditions proposed in Theorem \ref{thm_CICO}, which also guarantees the operational constraints of the system. The maximum bound on the magnitude of the admissible disturbance can be computed with the following optimization problem:}
\begin{equation}
\begin{aligned}
& \underset{\bar{z}\geq 0,\ \bar{u}\geq 0, \ \mu}{\text{maximize}} & & \mu \\
& \text{subject to} & & \gamma_{z,u}\bar{u}<(I-\gamma_{z,v}\gamma_\psi(\bar{z}))\bar{z} \\
& & & \gamma_{y,u}\bar{u}+\gamma_{y,v}\gamma_\psi(\bar{z})\bar{z} \leq \bar{y} \\
& & & \mu\leq c^T\bar{u}
\end{aligned}
\label{eqn_opt}
\end{equation}
 where $\bar{y}$ is the generator frequency limit provided by the system operators. The vector $c\in\mathbf{R}^n$ is used to fix the ratio of the disturbance entering at each bus. This procedure allows us to find the maximum disturbance magnitude at a particular bus, or alternatively, at a combination of buses.

\begin{proposition}
 {\color{black}
 The optimization problem \eqref{eqn_opt} is convex within the region defined by the angle deviation constraints $|\varphi^*_i|+\bar{z}_i\leq \pi, |\varphi^*_i|\leq \frac{\pi}{2}$ $\forall i$.
 }
\end{proposition}
\begin{proof}
Using the explicit expression for $\gamma_\psi(\bar{z})$ in the constraint $\gamma_{z,u}\bar{u}\leq(I-\gamma_{z,v}\gamma_\psi(\bar{z}))\bar{z}$, we obtain the following constraint:
\begin{equation}
\begin{aligned}
\gamma_{z,u}\bar{u}\leq(I &-\gamma_{z,v}\text{diag}(\cos\varphi^*))\bar{z} \\
 \hskip 2em &-\gamma_{z,v}\sin|\varphi^*|+\gamma_{z,v}\sin(|\varphi^*|+\bar{z})
\end{aligned}
\label{eqn_sin}
\end{equation}

The sinusoidal term is concave within the region defined by the bound $0\leq|\varphi^*_i|+\bar{z}_i\leq\pi$, and therefore the constraint in equation \eqref{eqn_sin} forms a convex region. Similarly, the constrained output condition is similarly bounded to a convex region of a sinusoidal function. Therefore, the constraints are convex, and we can conclude that the optimization problem \eqref{eqn_opt} is convex.
\end{proof}

\section{Simulations}
{\color{black} In this section, we numerically validate the theoretical and computational results presented in this paper. For illustration purposes, we first consider a single machine infinite bus system, on which we test and interpret the proposed results. Then, some practically important disturbance scenarios (e.g., simultaneous tripping of generators and loads, as well as the uncertainty from wind generation) will be tested on the standard IEEE 9-bus and 39-bus test cases.}

\subsection{Single Machine Infinite Bus (SMIB)}
The procedure and results are illustrated on a system composed of a single machine connected to an infinite bus through a lossless line. The dynamic equation is given by
\begin{equation}
M\ddot{\delta}+D\dot{\delta}+\phi\sin\delta=p+u
\end{equation}
where $M=1$, $D=1.2$, $p=0.2$ and $\phi=0.8$ are the parameters used in this study. For $u=0$, its equilibrium is given by $\delta^*=\arcsin(p/\phi)$, $\dot{\delta}=0$. Let the output be the frequency in Hertz, $y=\dot{\delta}/2\pi$. Substituting $z=\delta-\delta^*$, and $v=\sin(\delta)-\cos(\delta^*)w$ we get
\begin{equation}
M\ddot{x}+D\dot{x}+\phi\cos(\delta^*)x+\phi v=u
\end{equation}

In frequency domain,
\begin{equation}
\begin{aligned}
Z(s)&=\frac{1}{M s^2+D s+\phi\cos(\delta^*)}\big[U(s)-\phi V(s)\big] \\
&=G_{w,u}U(s)+G_{w,v}V(s),
\end{aligned}
\end{equation}
and $Y(s)=sZ/2\pi$.

The gains corresponding to the transfer functions $G_{y,u}$, $G_{y,v}$, $G_{z,u}$, and $G_{z,v}$ are $\gamma_{y,u}=0.178$, $\gamma_{y,v}=0.142$, $\gamma_{z,u}=1.434$, and $\gamma_{z,v}=1.148$, respectively. Following the proposed procedure, the nonlinear gain is a function of the bound on the phase difference. This can be seen in Figure \ref{fig_max_perturbation_2bus}(a). 

\begin{figure}[t]
	\centering
	\includegraphics[width=3.2in]{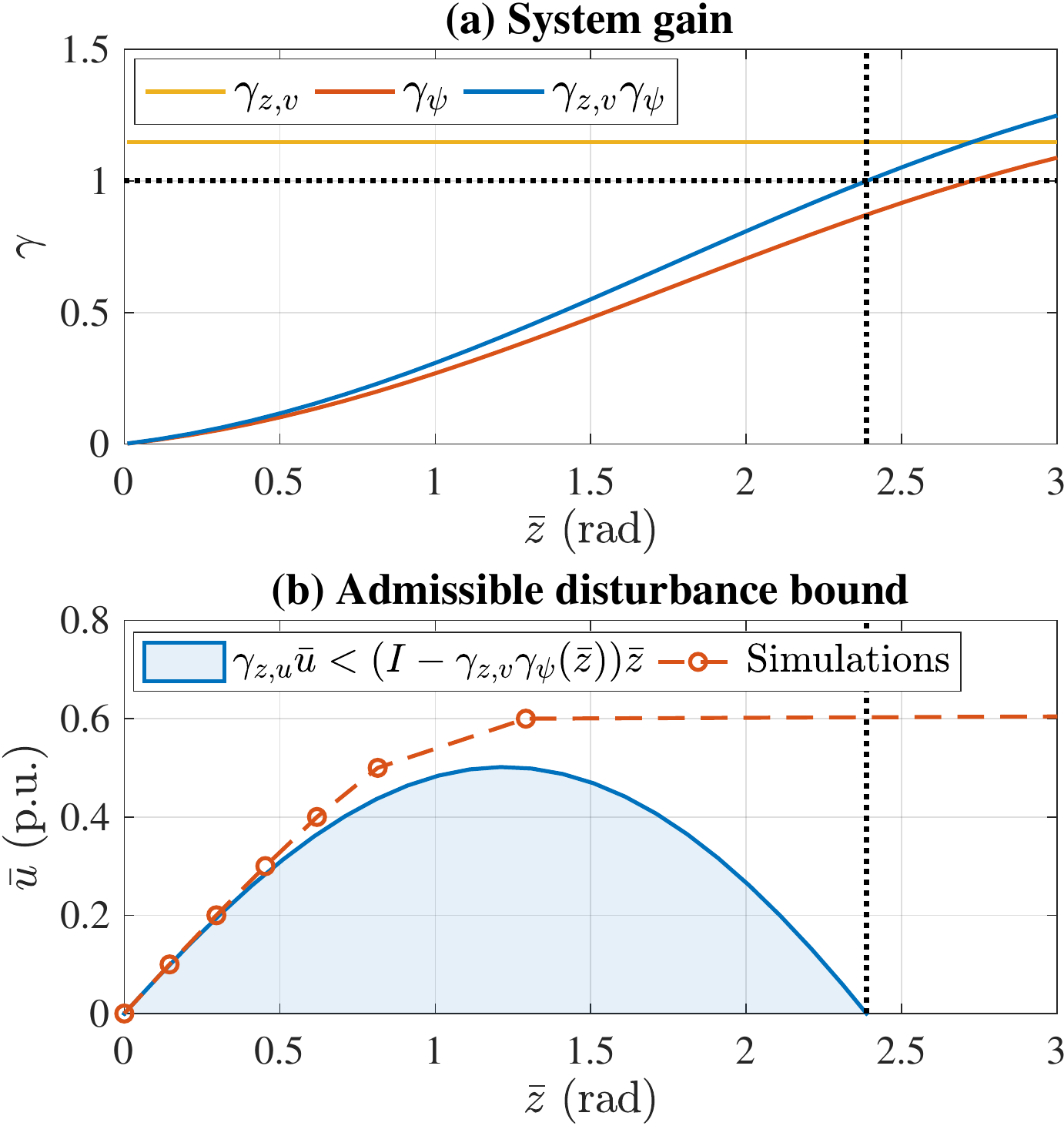}
	\caption{Maximum disturbance magnitude allowed as a function of sector bound for a SMIB system.}
	\label{fig_max_perturbation_2bus}
\end{figure}

Since the gain matrices are just scalars, the condition for BIBO stability is simply $\gamma_{z,v}\gamma_\psi<1$. In Figure \ref{fig_max_perturbation_2bus}(b) we plot with blue the CIBO stability condition presented in Theorem II. {\color{black} The vertical dashed black lines in the Figures \ref{fig_max_perturbation_2bus}(a) and \ref{fig_max_perturbation_2bus}(b) show that the small-gain condition is violated if the CIBO stability condition is not satisfied.}

{\color{black} In Figure \ref{fig_max_perturbation_2bus}(b), the estimation of the upper bound on the disturbance magnitude was computed by time-domain simulations. After applying a step disturbance with magnitude bounded by $\bar{u}$, the maximum phase difference deviation $\bar{z}$ was recorded. All the simulation points are represented with orange, and they are all connected by a dashed orange line. Since every simulation point is only a single realization among all possible disturbances, it only provides an upper-bound on the magnitude of the disturbance.}

The approach proposed in this paper uses convex optimization to efficiently compute the maximum magnitude for the admissible disturbance. Figure \ref{fig_max_perturbation_2bus}(b) shows that the gap between the upper-bound and the bound on the magnitude based on our method is very tight. The maximum disturbance magnitude allowed occurs when $\bar{z}$ is about 1.2 rad, which can be computed with the optimization problem \eqref{eqn_opt}. The small-gain condition in Figure \ref{fig_max_perturbation_2bus}(a) is violated when the angle deviation is about 2.4 rad. The bound on the disturbance magnitude becomes zero at the same $\bar{z}$, which illustrates the equivalence of condition (i) and (iii) in Lemma \ref{lemma_Zmatrix}. In Figure \ref{fig_max_freq_2bus}, the maximum frequency deviation is computed with the second condition in Theorem \ref{thm_CICO}. Similarly, a lower-bound on the frequency deviation was computed using the same procedure explained for the upper-bound on the magnitude of the disturbance.


\subsection{9-bus and 39-bus systems}
This section presents numerical case studies on the IEEE 9-bus and 39-bus systems. The nonlinear optimization was performed using the interior point method in IPOPT \cite{wachter06} on a PC laptop with an Intel Core I7 3.3 GHz CPU and 16GB of memory. In Figure \ref{fig_9and39}, we show a graphical representation of the computed maximum bound on the magnitude of the disturbance that can enter at every single bus. The results suggest that the bigger disturbances are allowed to enter at the buses with many neighbors to distribute the impact. For the generator nodes, the second order dynamics together with the governor reduce the damping ratio, and only small disturbances are admissible.

\begin{figure}[t]
	\centering
	\includegraphics[width=3.2in]{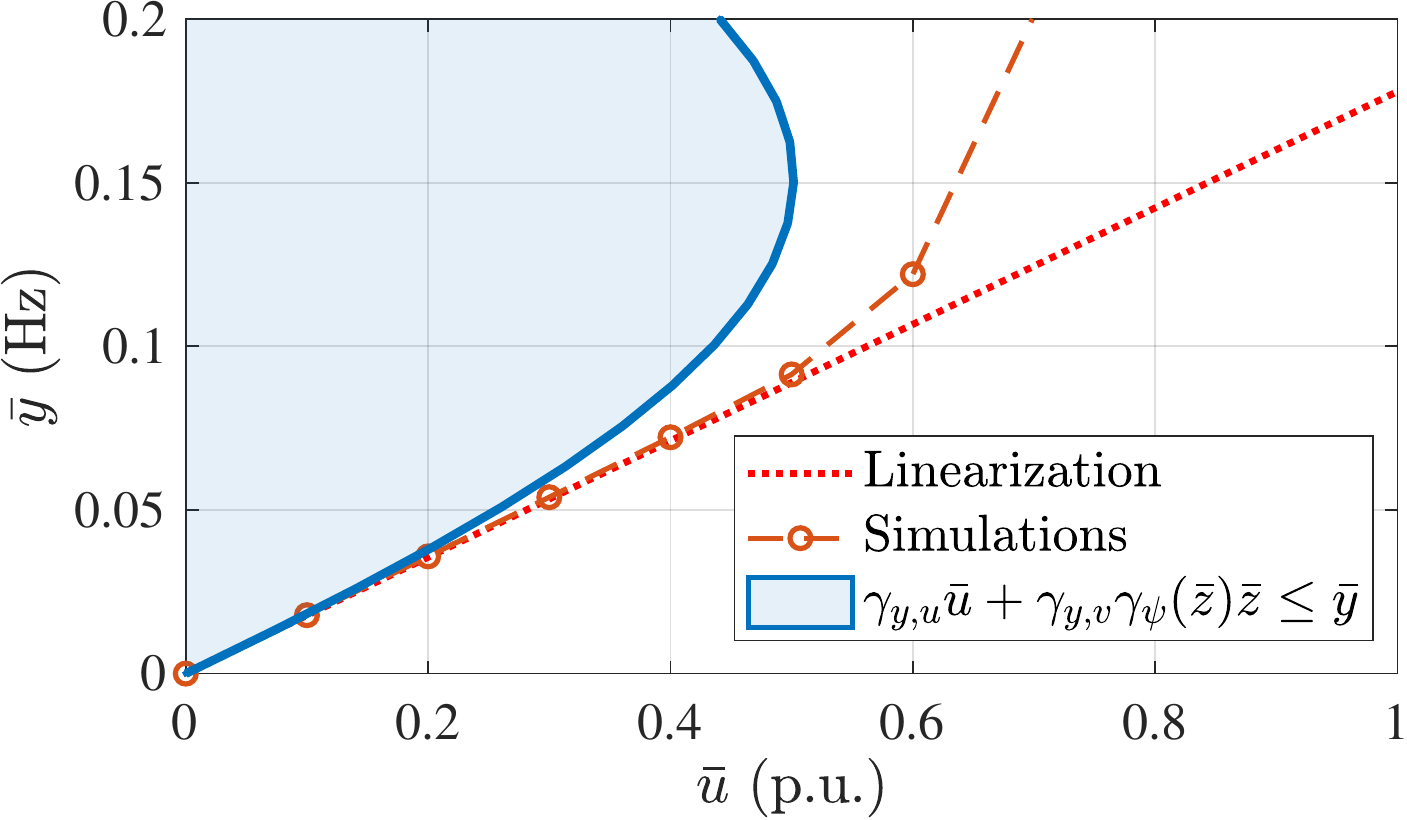}
	\caption{Maximum frequency deviation for a SMIB system.}
	\label{fig_max_freq_2bus}
\end{figure}

Regarding the computation time, for the 9-bus system the gain matrix took 1.86 seconds to compute, while the optimization took 0.017 seconds. For the 39-bus system, the computation time for the gain matrix was 166.9 seconds, while optimization time was 0.148 seconds. Therefore, the most computationally intense step in our method is the computation of gain matrix of the linear component, which requires simulation of an impulse response and numerical integration. However, the computation time of the gain matrix could be improved by estimating only an upper-bound, rather than its exact value \cite{balakrishnan92}.


For the 39-bus case study, we consider the following disturbance scenarios: a step disturbance to represent the simultaneous tripping of distributed generators, and a continuous disturbances to represent the varying power output from wind generation.


\subsubsection{Simultaneous Distributed Generators tripping}
In this scenario, we consider the simultaneous tripping of the loads at the buses 3, 15 and 27. The active power loads at those buses are 3.22 p.u., 3.2 p.u., and 2.81 p.u., respectively. Under the frequency constraint of 0.5 Hz, the maximum tripped load magnitude needs to be less than 0.939 p.u. Without the frequency constraint, the maximum disturbance magnitude allowed at each load is 2.29 p.u..

\subsubsection{Wind generation}
In this scenario, we consider the varying power output from wind generation at the buses 1, 9 and 16. Under the frequency constraint of 0.5 Hz, the deviation from the nominal generation needs to be less than 1.305 p.u. Without the frequency constraint, a deviation in the active power of 2.02 p.u. is allowed at each wind generator.

\section{Conclusion}
 Conventionally, operational constraints on the frequency deviation are not considered in the study of transient stability. The formulation presented in this paper offers a simple way to unify these considerations: transient stability and frequency constraints, and makes use of well-developed and efficient optimization methods to perform stability assessment. The input-output stability analysis provides a novel and practical solution to quickly identify the disturbances that the electric power grid can withstand, while never violating some imposed frequency constraints. The numerical study shows that our technique is not conservative, and can include a wide range of disturbances.

{\color{black} As future work, our results could be extended to consider additional features for the disturbance, such as ramping rate bound and duration. While our approach can deal with a very general class of disturbances, bounded only in magnitude, practical disturbances may come from a much more restricted class, with known characteristics. By exploiting such additional knowledge about the nature of disturbance, the methodology proposed here can be adapted to other specific applications. A different research direction could focus on designing robust controllers that would increase the magnitude of the admissible disturbances.}

\begin{figure}[t]
	\centering
	\includegraphics[width=3.2in]{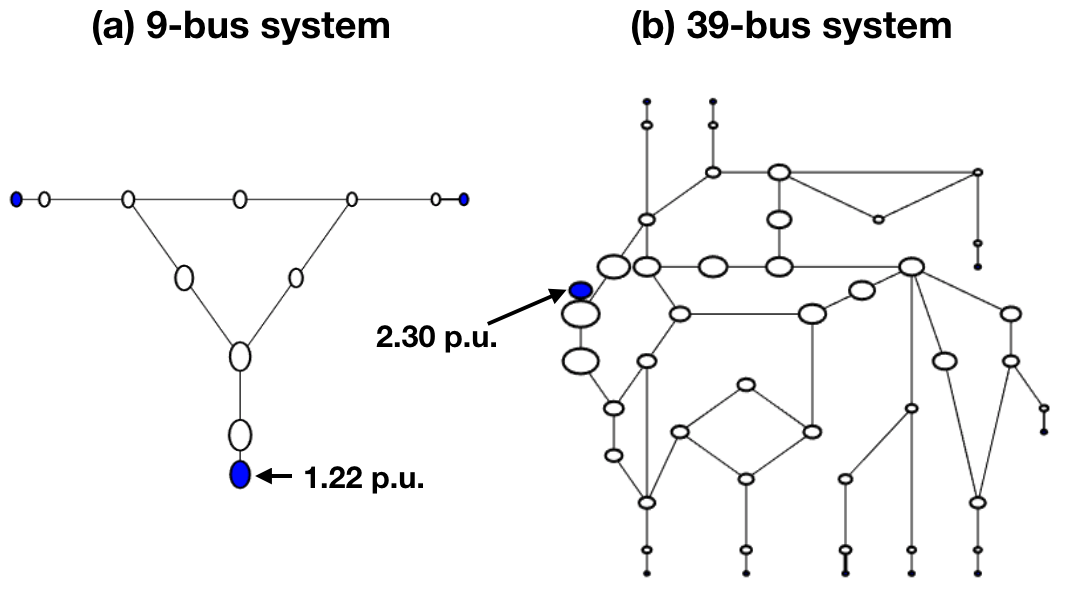}
	\caption{Maximum disturbance bound at every bus for the 9-bus and 39-bus systems. A disturbance on every individual node is considered and the resulting maximum bound is represented as the size of circle at that node. For both systems, a reference circle is labeled with its value.}
	\label{fig_9and39}
\end{figure}
\begin{figure}[t]
	\centering
	\includegraphics[width=3.6in]{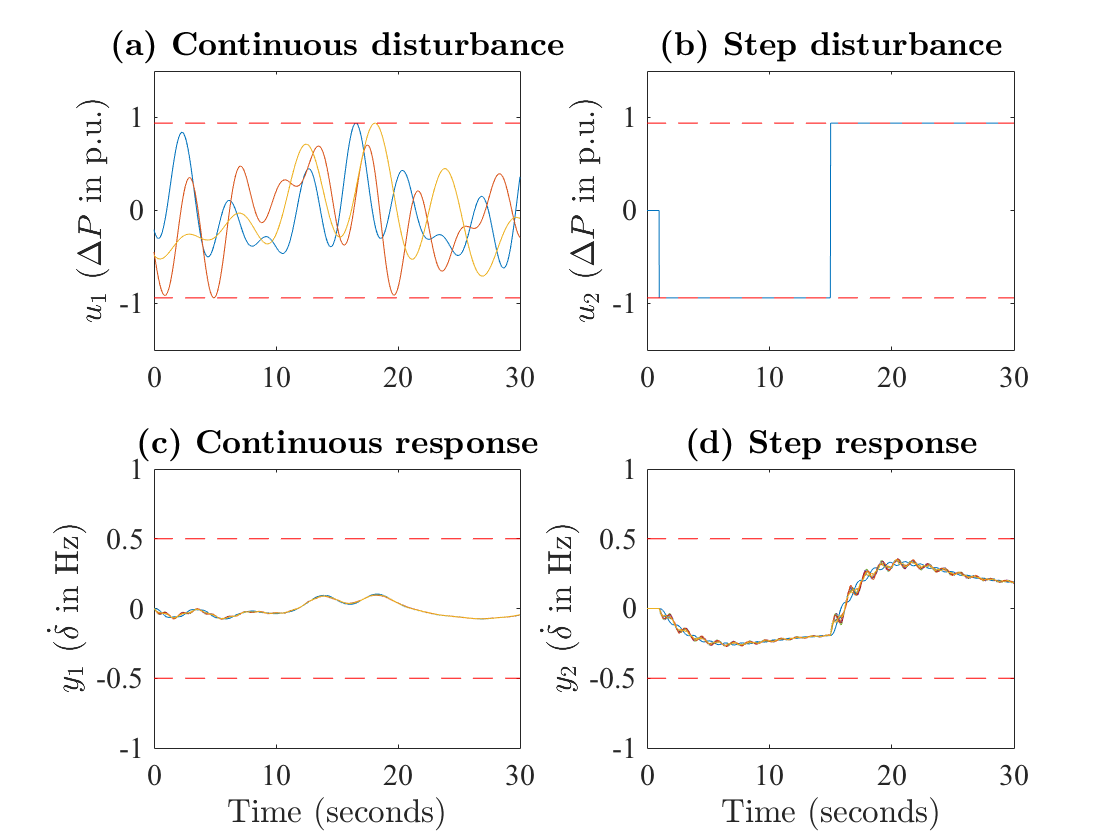}
	\caption{Simulation results for (a)varying wind generation, and (b) simultaneous generation tripping, together with their frequency response for the 39-bus system.}
	\label{fig_39bus}
\end{figure}


\ifCLASSOPTIONcaptionsoff
\newpage
\fi

\bibliographystyle{IEEEtran}
\bibliography{references}
\nocite{*}

\end{document}